\date{}

\date{}

\title{Finitely dependent random colorings of bounded degree graphs}
\author{\'Ad\'am  Tim\'ar }

\documentclass[11pt,naturalnames]{article}
\renewcommand\footnotemark{}
\usepackage[ruled,vlined]{algorithm2e}
\usepackage{amsmath}
\usepackage{amsthm}
\usepackage{amsfonts}
\usepackage{graphicx}
\newif\ifhyper\IfFileExists{hyperref.sty}{\hypertrue}{\hyperfalse}
\ifhyper\usepackage{hyperref}\fi
\usepackage{enumitem}
\usepackage{subfig}
\usepackage{caption}
\usepackage{keyval}
\usepackage[margin=1in]{geometry}
\usepackage{setspace}
\usepackage[displaymath, mathlines,pagewise]{lineno}
\theoremstyle{definition}

\newtheorem{theorem}{Theorem}
\newtheorem{lemma}[theorem]{Lemma}

\newtheorem{remark}[theorem]{Remark}

\def \proof {{ \medbreak \noindent {\bf Proof.} }}
\def\proofof#1{{ \medbreak \noindent {\bf Proof of #1.} }}
\def\proofcont#1{{ \medbreak \noindent {\bf Proof of #1, continued.} }}
\def\supp{{\rm supp}}
\def\max{{\rm max}}
\def\min{{\rm min}}
\def\dist{{\rm dist}}
\def\Aut{{\rm Aut}}
\def\id{{\rm id}}
\def\Stab{{\rm Stab}}

\begin{document}
\maketitle
\let\thefootnote\relax\footnotetext{\footnotesize{Partially supported by Icelandic Research Fund Grant 239736-051 and ERC
grant No. 810115-DYNASNET.}}

\bigskip

\def\eref#1{(\ref{#1})}
\newcommand{\Prob} {{\bf P}}
\newcommand{\C}{\mathcal{C}}
\newcommand{\LL}{\mathcal{L}}
\newcommand{\Z}{\mathbb{Z}}
\newcommand{\N}{\mathbb{N}}
\newcommand{\HH}{\mathbb{H}}
\newcommand{\Rr}{\mathbb{R}^3}
\newcommand{\h}{\mathcal{H}}
\def\diam{\mathrm{diam}}
\def\length{\mathrm{length}}
\def\ev#1{\mathcal{#1}}
\def\Isom{{\rm Isom}}
\def\Re{{\rm Re}}
\def \eps {\epsilon}
\def \P {{\Bbb P}}
\def \E {{\Bbb E}}
\def \proof {{ \medbreak \noindent {\bf Proof.} }}
\def\proofof#1{{ \medbreak \noindent {\bf Proof of #1.} }}
\def\proofcont#1{{ \medbreak \noindent {\bf Proof of #1, continued.} }}
\def\supp{{\rm supp}}
\def\max{{\rm max}}
\def\min{{\rm min}}
\def\dist{{\rm dist}}
\def\Aut{{\rm Aut}}
\def\id{{\rm id}}
\def\Stab{{\rm Stab}}

\newcommand{\lra}{\leftrightarrow}
\newcommand{\xlra}{\xleftrightarrow}
\newcommand{\xnlra}{\xnleftrightarrow}
\newcommand{\pc}{{p_c}}
\newcommand{\pt}{{p_T}}
\newcommand{\ptk}{{\hat{p}_T}}
\newcommand{\pl}{{\tilde{p}_c}}
\newcommand{\pe}{{\hat{p}_c}}
\newcommand{\pr}{\mathrm{\mathbb{P}}}
\newcommand{\pp}{\mu}
\newcommand{\ex}{\mathrm{\mathbb{E}}}
\newcommand{\ee}{\mathrm{\overline{\mathbb{E}}}}

\newcommand{\om}{{\omega}}
\newcommand{\ebd}{\partial_E}
\newcommand{\ivbd}{\partial_V^\mathrm{in}}
\newcommand{\ovbd}{\partial_V^\mathrm{out}}
\newcommand{\q}{q}
\newcommand{\TT}{\mathfrak{T}}
\newcommand{\T}{\mathcal{T}}
\newcommand{\RR}{\mathcal{R}}

\newcommand{\CC}{\Pi}
\newcommand{\BB}{\Pi}

\newcommand{\A}{\mathcal{A}}
\newcommand{\cc}{\mathbf{c}}
\newcommand{\pa}{{PA}}
\newcommand{\degi}{\deg^{in}}
\newcommand{\dego}{\deg^{out}}
\def\Pn{{\bf P}_n}

\newcommand{\R}{\mathbb R}
\newcommand{\F}{F}
\newcommand{\FF}{\mathfrak{F}}
\newcommand{\Can}{\rm Can}
\newcommand{\Vol}{\mathrm Vol}

\def\UST{{\rm UST}}
\def\Gstar{{{\cal G}_{*}}}
\def\Gstarstar{{{\cal G}_{**}}}
\def\Gstarplus{{{\cal G}_{*}^\frown}}
\def\Rel{{\cal R}}
\def\Comp{{\rm Comp}}
\def\calG{{\cal G}}
\def\omps{{\omega_\delta^\eps}}

\begin{abstract}
%We prove that every infinite graph of degree at most $d$ has a 4-dependent random proper $4^{d(d+1)/2}$-coloring. Using $5^{d(d+1)/2}$ colors, we can also make it a finitary factor of iid coloring.
We prove that every (possibly infinite) graph of degree at most $d$ has a 4-dependent random proper $4^{d(d+1)/2}$-coloring, and one can construct it as a finitary factor of iid. For unimodular transitive (or unimodular random) graphs we construct an automorphism-invariant (respectively, unimodular) 2-dependent coloring by $3^{d(d+1)/2}$ colors. In particular, there exist random proper colorings for $\Z^d$ and for the regular tree that are 2-dependent and automorphism-invariant, or 4-dependent and finitary factor of iid.
\end{abstract}

\bigskip

Consider a process (random decoration) $(X_i)_{i\in V(G)\cup E(G)}$ of a graph $G$, with $X_i\in {\mathcal X}$ for some finite set ${\mathcal X}$. Decorating only subsets of the edges or vertices also fits in this framework (add a new symbol to ${\mathcal X}$ if necessary, and assign it to all the edges and vertices that were not decorated before).
Say that the random decoration is {\it $k$-dependent}, if $(X_i)_{i\in A}$ and $(X_i)_{i\in B}$ are independent from each other for every two subgraphs $A,B\subset G$ at distance greater than $k$ from each other. A random decoration is {\it finitely dependent} if it is $k$-dependent for some $k$. The process $(X_i)_{i\in V(G)\cup E(G)}$ is a {\it factor} of process $(Y_i)_{i\in V(G)\cup E(G)}$ if it can be constructed as a measurable function from $(Y_i)_{i\in V(G)\cup E(G)}$ that commutes with all isomorphisms of $G$. The process $(X_i)_{i\in V(G)\cup E(G)}$ if a {\it factor of iid} if it is a factor of the process $(Y_i)_{i\in V(G)}$ with $Y_i$ independent uniform $[0,1]$ random variables. In practice this means that we can construct $Y_i$ from a large enough $X_j$-labelled neighborhood of $i$, up to an arbitrarily small error. A factor is {\it finitary} if there is no such error when the neighborhood is large enough. More precisely, for every $i\in V(G)\cup E(G)$, $Y_i$ is determined by some bounded neighborhood of $i$ (which neighborhood may be random, depending on the $X_j$'s in it).
A random decoration is a {\it $b$-block factor} if it arises as a factor of iid that only uses the $b$-neighborhood of every vertex. A $b$-block factor is also called a {\it block factor}. Our main result is the following.
 
\begin{theorem}\label{main}
Let $G$ be a graph of maximal degree $d$. Then there is a $4$-dependent finitary factor of iid proper coloring of
$G$ with at most $4^{d(d+1)/2}$ colors. 
\end{theorem}

If $G$ is a transitive graph, every factor of iid decoration of $G$ has an automorphism-invariant distribution. Imposing this weaker restriction of automorphism-invariance on finitely dependent colorings, we get another existence result, as in the next theorem. For simplicity we restrict ourselves to the class of unimodular transitive graphs (which contains every locally finite Cayley graph, such as the regular tree or the nearest neighbor graph of $\Z^d$). But we will also make the treatment more general, by allowing all unimodular random graphs (URG's), not just transitive ones. Then we will be looking for colorings that are jointly unimodular with the graph (i.e., the decorated graph coming from the coloring is unimodular), which is the standard generalization of automorphism invariance to this class.
For the definition and basic properties of URG's, see \cite{AL}. %This generality will be needed for our induction proof, but the claim could be extended further. %, to include automorphism-invariant colorings of nonunimodular transitive graphs, for example. 
We decided to draw the line between generality and efficiency by picking the class of unimodular random graphs, even though we are not fully using their definition, but rather the fact that the class of URG's is closed under factor of iid transformations. In Remark \ref{general} we look into the possibility of further generalizations.

%More generally, we will consider {\it stationary random graphs}, i.e. rooted graphs (considered to be equivalent when they are rooted isomorphic) with a probability distribution that is invariant under rerooting by a simple random walk step, and consider colorings that are invariant under this rerooting operation ({\it invariant} colorings). Among examples of stationary random graphs are transitive graphs and unimodular random graphs. See \cite{BC} for the more thorough definition and background.

%For the proof we have to leave the class of transitive graphs. Say that a bounded degree graph is {\locally constructible} if one can apply 

\begin{theorem}\label{3color}
Let $G$ be a unimodular random graph of maximal degree $d$.
Then there is a unimodular random coloring of $G$ by $3^{d(d+1)/2}$ colors which is 2-dependent almost surely, conditional on $G$. In particular, every unimodular transitive graph has an automorphism-invariant 2-dependent $3^{d(d+1)/2}$-coloring.
\end{theorem}

Our proof relies on the following theorem of Holroyd and Liggett \cite{HL}, and of Holroyd \cite{H}. The first part of the claim is from \cite{HL} while the last assertion is from \cite{H}.

\begin{theorem}[\cite{HL}, \cite{H}]\label{HL}
There is a stationary 1-dependent 4-coloring and a 2-dependent 3-coloring for $\Z$.
Moreover, there is a 1-dependent 4-coloring that is a finitary factor of iid.
\end{theorem}

Another result we will need for the proof of Theorem \ref{3color} is by Holroyd, Hutchcroft and Levy:

\begin{theorem}[\cite{HHL}]\label{HHL}
There is an automorphism-invariant 1-dependent 4-coloring and a 2-dependent 3-coloring for every cycle of length at least 3.
Moreover, there is a 1-dependent 4-coloring that is a finitary factor of iid.
\end{theorem}

%The question whether there exists a finitely dependent isometry-invariant random coloring has been widely open, with the case of $\Z$ settled in ..., and that of $\Z^d$ in \cite{H2}. While these also follow from our theorem, all the proofs, including ours, are reliant on the theorem of Holroyd and Liggett. 
It is clear that there exists no finitely dependent 2-coloring for $\Z$. 
Schramm has shown that there is no 1-dependent 3-coloring for $\Z$, see \cite{HSW}. A block factor is always finitely dependent, but it is impossible to color $\Z$ (and hence any infinite graph) as a block factor, see \cite{HSW}. Whether there is any finitely dependent process that is not a block factor was first asked in \cite{IL}, and answered in the negative by Burton, Goulet and Meester \cite{BGM}. The first really natural example for a finitely dependent process that cannot arise as a block factor was the one in Theorem \ref{HL} by Holroyd and Liggett, who found finitely dependent colorings of $\Z$. The construction did not come from a factor of iid, but turned out to be a finitary factor of iid in \cite{H}, as stated in Theorem \ref{HL}. Having settled the problem for $\Z$, a coloring of $\Z^d$ can be defined by the standard trick of coordinate-wise coloring (as in \cite{HL}), but this only gives a translation-invariant result, because there is no way to list the axes in an an isometry-invariant way (which is essential in taking a product coloring). 

A recent preprint of Holroyd \cite{H2} solves the problem of a finitely dependent factor of iid 4-coloring, with the dependence radius at least exponential in $d!$. Our Theorem \ref{main} implies that there is a 4-dependent such coloring, but a larger number of colors is used. Holroyd asks in \cite{H2} if there is a finitely dependent finitary factor of iid coloring of the regular tree, which we settle in Theorem \ref{main}. Liggett and Tang called the problem of finding an automorphism invariant 1-dependent coloring for the regular tree one main open problem in the field (\cite{LT}), mentioning some implications. We solve this question with 2-dependence instead of 1 in Theorem \ref{3color}.

%. The question of finitely dependent colorings for a regular tree is listed among the Open questions in \cite{H2}. Besides an answer to this, our main result implies the case of $\Z^d$ as in \cite{H2}, and those of cycles \cite{HHL}, but with worse dependence (???). Similarly to our proof, these papers also rely on Theorem \ref{HL}.

\begin{proofof}{Theorem \ref{3color}}
We prove by induction on $d$. For $d=1$ the claim is trivial. Suppose that for every graph of degrees at most $d-1$ there is a 2-dependent unimodular coloring by at most $3^{(d-1)d/2}$ colors.

Now let $G$ be a graph of degrees at most $d$. We will give a proper coloring by color set $\{1,2,3\}^{d(d+1)/2}$. We may assume that $G$ has no isolated vertices (otherwise color each of them by a uniform random color and remove them). For every vertex $x$, define a function $o_x$ from the set of neighbors of $x$ to $\{1,\ldots,d\}$ such that $o_x$ is injective, but otherwise chosen uniformly at random.
%First we define a directed graph (digraph) $D'$ on some subset of the edges of $G$ such that every vertex is in some edge of $D'$ and the outdegree of every vertex is at most 1. To do so, 
Let every vertex $x$ randomly choose one of its neighbors, $h(x)$, and consider the digraph $D=(V(G),\{(x,h(x))\})$. Every vertex of $D$ has outdegree 1. Assign a {\it label} to every edge $\{(x,h(x))\}\in E(D)$, namely, let the label be $i$ if $o_{h(x)}(x)=i$. 

Now, pick a label $i\in\{1,\ldots,d\}$. All edges pointing to a vertex $v$ have distinct labels (since $o_v$ is injective), and $v$ has outdegree at most 1 in $D$, hence the number of edges of label $i$ at $v$ is at most 2. The digraph $D_i\subset D$ induced by label-$i$ edges is a vertex-disjoint union of directed paths and directed cycles. Apply Theorems \ref{HL} and \ref{HHL} to each of these components to get a 2-dependent unimodular coloring $c_i$ of $D_i$ by colors $\{1,2,3\}$. 
(Although Theorem \ref{HL} is for $\Z$, the stationary coloring can clearly be defined for subpaths of it as well, which is then automatically a unimodular decoration. We mention that it is only a matter of choice whether here we rely on the stationary coloring from the theorem for these directed paths of the $D_i$, or forget about the orientation and use the fact from \cite{HL} that these colorings are also automorphism-invariant. We also mention that for directed cycles of length 2, where Theorem \ref{HHL} does not formally apply, we can delete one of the edges at random and color the resulting path using Theorem \ref{HL}, which trivially gives a coloring for the original 2-cycle as well.) 
For vertices $x\in V(G)\setminus V(D_i)$, define $c_i(x)$ to be uniform in $\{1,\ldots,d\}$.
Now, define a color $c (x)=(c_1(x),\ldots,c_d(x))\in \{1,2,3\}^d$ for every $x\in V(T)$. Suppose $x$ and $y$ are adjacent vertices in $D$. Then the edge between them belongs to some $D_i$, and so the $i$'th coordinates of $c(x)$ and $c(y)$ are different. So $c$ gives a proper coloring of $D$, which is a unimodular decoration of $G$. It is 2-dependent (when viewed as a decoration of $G$), for the following reason. 
To tell the edges $E(D_i)|_A$ of $D_i$ incident to some given $A\subset V(G)$, it is enough to look at the choices (the $o_x$ and the $h(x)$) made in a 1-neighborhood of $A$. So $E(D_i)|_A$ and $E(D_i)|_B$ are independent if $A$ and $B$ have distance greater than 2. The 3-coloring of the $D_i$ is also 2-dependent, so the entire process remains 2-dependent.

Every vertex of $G$ is incident to some edge in $D$, hence if we remove the (undirected version of the) edges of $D$ from $G$, we get a new graph of maximal degree at most $d-1$. This graph is unimodular, since we obtained it from the unimodular random graph $G$
%through a factor of iid modification. 
by removing a (jointly) unimodular subgraph of it.
Apply the induction hypothesis to color every vertex of this forest by $3^{(d-1)d/2}$ colors in a 2-dependent way. Using this color as a first coordinate, and the color given in $D$ as a second coordinate, we get a unimodular 2-dependent proper coloring by $3^{d(d+1)/2}$ colors.
\end{proofof}
\qed

\begin{remark}\label{general}
We were using two properties of unimodular (decorated) random graphs in the above proof, rather than the full definition. One property is that subgraphs ``coming from'' jointly unimodular decorations (such as $D$ or its ``complement'' in $G$) are also unimodular. The other one is that the coloring of the components of the $D_i$'s originated from the automorphism-invariant coloring of Theorems \ref{HL} and \ref{HHL}, and hence the final coloring is also unimodular. One could extend the theorem further, e.g. look at stationary colorings of stationary random graphs.
\end{remark}

%\begin{remark}
%We could have avoided the use of Theorem \ref{HHL} by an extra argument, whose sketch is the following. From each directed cycle of edges of the same label in $D$ delete an edge that connects a local maximum (with regard to an iid label) in this cycle to the vertex where its outgoing edge is pointing to. The resulting digraph $D'$ has no directed cycles (hence we can color the labelled paths using Theorem \cite{HL}), and still contains an edge from each vertex of $G$ (so we can apply the induction hypothesis after undirecting and removing the edges of $D'$ from $G$). 
%\end{remark}

%We did not strive to optimize the number of colors given by our contruction. However, for the case of a 3-regular graph (or graph of degrees at most 3), a more careful analysis gives the following corollary. We note that in this case $E(D')$ is labelled by two labels, hence it is partitioned into two collections of paths. The number of colors needed to color $D'$ is hence $4^2$

%\begin{corollary}

%\end{corollary}

Our proof for Theorem \ref{main} will follow closely that of Theorem \ref{3color}, but we need to turn the construction into a factor of iid. This results in two changes. In Theorem \ref{HL} only a 4-coloring is available as a factor of iid, hence we need to use 4 colors for the $D_i$. Secondly, there is no factor of iid version of the finitely dependent coloring for cycles, therefore, instead of using Theorem \ref{HHL}, we will remove some edges from $D$ to get rid of all the cycles, but in a way that the induction hypothesis still applies for the graph outside of it.

\begin{lemma}\label{cycle}
Let $D$ be a digraph where every vertex has outdegree 1. Then there is a 2-block factor subgraph $D'$ of $D$ with no directed cycles and no isolated vertices.
\end{lemma}

\begin{proof}
Consider the components $D_1,D_2,\ldots$ of $D$ (that is, the sets $V(D_i)$ form a partition of $V(D)$, with no edge between any two of them). Every $D_i$ either has no directed cycle, or it contains a single directed cycle $C_i$. To see this, 
suppose by contradiction that there are two directed cycles in $D_i$. Consider a minimal connected subgraph $H$ of $D_i$ that contains both of these cycles. Then $H$ is necessarily finite. The sum of outdegrees in $H$ is at most $|V(H)|$. This has to be equal to the number of edges in $H$. But that is at least $|V(H)|+1$, because there are at least two edges outside of a spanning tree of $H$ (otherwise $H$ would only contain one cycle).  

Denote by $h(x)$ the head of the directed edge in $D$ whose tail is $x$. Let $\{\xi (x):x\in V(D)\}$ be iid $[0,1]$ labels. For each vertex $x$ of indegree at least 1, {\it mark the edge} $(x,h(x))$ if $\xi(h(x))>\xi(x)$ and $\xi(h(x))>\xi(h(h(x)))$ (i.e., if $h(x)$ is a {\it local maximum} on the directed path starting from $x$). Remove all marked edges from $D$ to define $D'$. If there exists a directed cycle in $D'$, then a vertex $x$ of maximal $\xi$-value in the cycle is a local maximum, and hence the edge $(x,h(x))$ would have been marked and not present in $D'$. So there are no directed cycles in $D'$. On the other hand note that the head of a marked edge $(x,h(x))$ cannot be the tail of a marked edge $(h(x),h(h(x))$, because that would imply $\xi(h(x))>\xi(h(h(x))$ and $\xi(h(x))<\xi(h(h(x))$ at the same time. So whenever a vertex of indegree at least 1 is incident to a marked edge, it is also incident to an edge that is not marked. If a vertex has indegree 0, then its outgoing edge cannot be marked by definition. This shows that $D'$ has no isolated vertices.
\end{proof}
\qed

We will assume throughout the next proof that every vertex of $G$ comes with a random variable that is uniform in $[0,1]$, and these are independent. It is a standard fact that from one such random variable one can obtain infinitely many iid uniform $[0,1]$ random variables, for example by taking the sequence of $2^k$'th digits in a binary expansion for $k=1,2,\ldots$. So we can define several factor of iid functions on $G$, and still maintain that they are jointly factor of iid. 

\begin{proofof}{Theorem \ref{main}}
We follow the proof of Theorem \ref{3color} to the point when $D$ has been constructed. For every $x\in V(G)$, the random choices of the $o_y$ and the $h(y)$ over the neighbors $y$ of $x$ determine the edges of $D$ that contain $x$, and also their labels. So $D$ with its edge labels can be attained as a 2-dependent factor of iid.
Apply Lemma \ref{cycle} to get a digraph $D'$ on $V(D)=V(G)$, which has no directed cycle or isolated vertices. Every vertex of $D'\subset D$ has outdegree at most 1. Since the construction of $D'$ can be done as a 2-block factor, $D'$ is a 4-dependent finitary factor of iid from $G$. %Consider the digraph $D_i\subset D'$ induced by label-$i$ edges. It is a vertex-disjoint union of directed paths. Apply Theorem \ref{HL} to each of the components to get a 1-dependent finitary factor of iid coloring $c_i$ of $D_i$ by colors $\{1,2,3,4\}$. (Although Theorem \ref{HL} is for $\Z$, one can apply it to finite or half-infinite directed paths as well, by extending the path to a biinfinite one, with the newly added vertices getting their iid labels from an endpoint of the path... Then, once we have the coloring, we can just forget about the newly added vertices.) 

Define the digraph $D_i\subset D'$ as the sub-digraph induced by label-$i$ edges in $D'$. It is a disjoint union of paths, hence we can apply Theorem \ref{HL} to each of its components. (Although Theorem \ref{HL} is for $\Z$, one can apply it to finite or half-infinite directed paths as well, by extending the path to a biinfinite one, with the newly added fictional vertices getting their iid labels from an endpoint of the path... Then, once we have the coloring, we can just forget about the fictional vertices.) We get a 1-dependent finitary factor of iid 4-coloring $c_i$ of $D_i$ by colors $\{1,2,3,4\}$. For vertices $x\in V(G)\setminus V(D_i)$, define $c_i(x)$ to be uniform in $\{1,2,3,4\}$ (defined as a factor).
As before, define a color $c (x)=(c_1(x),\ldots,c_d(x))\in \{1,2,3,4\}^d$ for every $x\in V(G)$.
Every vertex of $G$ is incident to some edge in $D'$, hence 
if we remove the (undirected version of the) edges of $D'$ from $G$, we can apply the induction hypothesis. 
Using the color in this graph as a first coordinate, and the color given in $D'$ as a second coordinate, we get a proper coloring by $4^{d(d+1)/2}$ colors as a finitary 4-dependent factor of iid.
\end{proofof}
\qed

\begin{remark}
The colorings given in the above proofs are not invariant under the permutation of colors. This invariance fails in the induction step, the rest is symmetric under permutations (using the fact that the colorings in Theorems \ref{HL} and \ref{HHL} have this invariance property). A symmetrization of the above coloring rule is work in progress.
\end{remark}

\bigskip

\noindent
{\bf Acknowledgments:} I thank P\'eter Mester and G\'abor Pete for helpful discussions.

\ \\
\newpage
{\bf \'Ad\'am Tim\'ar}\\
Division of Mathematics, The Science Institute, University of Iceland\\
Dunhaga 3 IS-107 Reykjavik, Iceland\\
and\\
HUN-REN Alfr\'ed R\'enyi Institute of Mathematics\\
Re\'altanoda u. 13-15, Budapest 1053 Hungary\\
\texttt{madaramit[at]gmail.com}\\


\begin{thebibliography}{AAA}


\bibitem{AL}
Aldous, D., Lyons, R. (2007) 
Processes on unimodular random networks
{\it Electron. J. Probab.} {\bf 12}, 1454-1508.

\bibitem{BGM}
Burton, R.M., Goulet, M. and Meester, R. (1993) On 1-dependent processes and
k-block factors {\it Ann. Probab.}, {\bf 21} 2157–2168.

\bibitem{H} Holroyd, A.E. (2017) One-dependent coloring by finitary factors {\it Ann. Inst. H. Poincaré Probab. Statist.} {\bf 53} (2) 753 - 765.

\bibitem{H2} Holroyd, A.E. (2023) Symmetrization for finitely dependent colouring, preprint, {\tt arXiv:2305.13980}

\bibitem{HHL} Holroyd, A.E., Hutchcroft, T. and Levy, A. (2018) Finitely dependent cycle coloring
{\it Electronic Communications in Probability}, {\bf 23} 1–12.

\bibitem{HL} Holroyd, A.E. and Liggett, T.M. (2016) Finitely dependent coloring {\it Forum of Mathematics, Pi} (Vol. 4, p. e9). Cambridge University Press.

\bibitem{HSW} Holroyd, A. E., Schramm, O. and Wilson, D.B. (2017) Finitary coloring {\it Annals
of Probability}, {\bf 45} 2867–2898.

\bibitem{IL} Ibragimov, I. A. and Linnik, J. V. (1965) {\it Nezavisimye stalionarno svyazannye
velichiny} Izdat. “Nauka”, Moscow, 1965.

\bibitem{LT} Liggett, M.T. and Tang, W. (2023) One-dependent colorings of the star graph {\it Ann. Appl. Probab.}, {\bf 33} 4341-4365,


\end{thebibliography}
\end{document}